\documentclass[letterpaper, 11 pt,onecolumn]{IEEEtran}

\usepackage{graphicx,epstopdf}
\usepackage{amsmath,amssymb,amsthm,authblk}
\usepackage{amsmath,amssymb,authblk}
\usepackage[acronym]{glossaries}
\usepackage{color}
\usepackage{tikz} 
\usepackage[english]{babel}
\usepackage{comment}
\usepackage{subcaption}
\usepackage{mathtools}

\newtheorem{theorem}{Theorem} 
\newtheorem{lemma}[theorem]{Lemma}

\newtheorem{proposition}[theorem]{Proposition} 

\newtheorem{example}{Example}
 \newacronym{tas}{TASEP}{Totally Asymmetric Simple Exclusion Process}
\newacronym{rfm}{RFM}{Ribosome Flow Model}

\newcommand*\diff{\mathop{}\!\mathrm{d}}
 
\newcommand {\R}{\mathbb R}

 \newcommand{\ave}{\operatorname{ave}}
 
\newcommand{\rev}[1]{{{#1}}}

\IEEEoverridecommandlockouts

\title{	No Switching Policy is Optimal for a \rev{Positive} 
	Linear System with a  Bottleneck~Entrance
	}

\author{Mahdiar Sadeghi$^1$, M. Ali Al-Radhawi$^{1}$, Michael Margaliot$^2$, and Eduardo D. Sontag$^{1,3}$
	\thanks{$^1$Departments of Electrical and Computer Engineering and of Bioengineering, Northeastern University, Boston, MA 02115, USA. E-mails: \textit{sadeghi.ma@husky.neu.edu}, \textit{malirdwi@\{northeastern,mit\}.edu}, \textit{sontag@sontaglab.org} }
	\thanks{$^2$Department of Electrical Engineering-Systems, Tel Aviv University, Tel Aviv-Yafo, Israel.
		E-mail: \textit{michaelm@tauex.tau.ac.il}  }%
		\thanks{$^3$Laboratory of Systems Pharmacology, Harvard Medical School,
			Boston, MA 02115.    }
}

\begin{document}

	\maketitle
	\thispagestyle{empty}
	\pagestyle{empty}

	\begin{abstract}
	We consider a nonlinear SISO system that is a  cascade of a scalar 
	``bottleneck entrance'' and an arbitrary Hurwitz  positive  linear system. 
	This system entrains i.e. in response to a $T$-periodic inflow every solution converges to a unique~$T$-periodic solution of the system. 
We study the problem of maximizing the averaged throughput  
 via controlled switching. The objective is to choose a periodic inflow rate with a given mean value that maximizes the averaged outflow  rate of the system.  We compare two strategies: 
1)~switching between a high and low value, 
and
2)~using a constant inflow equal to the prescribed mean value. We show that no switching policy can outperform a constant inflow rate, though it can approach it asymptotically. 
We describe several potential applications
of this problem  in traffic systems, ribosome flow models, and scheduling at security checks. 
	\end{abstract}

\begin{IEEEkeywords}
		Entrainment,
		switched systems,
		airport security, 
		ribosome flow model,
		traffic systems.
\end{IEEEkeywords}
	
	\section{Introduction} \label{sec:intro}
	
		\IEEEPARstart{M}{aximizing} the throughput is crucial in a wide range of nonlinear applications such as traffic systems, ribosome flow models~(RFMs), and scheduling at security checks. 
		We model the  occupancy at time~$t$ in such applications by the 
		normalized state-variable~$x(t) \in [0,1]$.
		  In traffic systems,~$x(t)$ can be interpreted as the number of vehicles relative to the maximum capacity of a highway segment. 
			In  biological transport models,  $x(t)$ can be
			interpreted as the probability that a biological ``machine'' (e.g. ribosome,   motor protein) 
			 is bound to a specific segment of the ``trail'' it is traversing  (e.g. mRNA molecule, filament) at time~$t$.
			For the security check, it is the number of passengers at a security gate relative to its capacity.
		 
		The output in such systems is a nonnegative
		outflow which can be interpreted as the rate  of cars exiting the highway for the traffic system, ribosomes leaving the~mRNA for the~RFM~\cite{margaliot2012stability}, and passengers leaving the gate for the security check.  
		The inflow rates are often periodic, such as  those controlled by traffic light signals,
		the periodic cell-cycle division process, or periodic flight schedules.  Proper functioning often 
		requires \emph{entrainment} to such excitations i.e. internal
		processes  must operate in a periodic pattern with the same period as the excitation~\cite{Glass1979}. 
		In this case, in response   to a~$T$-periodic inflow the outflow converges to a~$T$-periodic
		pattern, and the  
		\textit{throughput} is then defined as the ratio of the average outflow relative 
		to the average inflow over the period~$T$.         
		
	As a general model for studying such applications, 
		we consider the cascade of two systems   shown in Fig.~\ref{f.cascade}. 
		The first block is called the \emph{bottleneck} and is given by:
			\begin{align} \label{eq:con}
		\dot{x}(t) &= \sigma(t) (1-x(t)) - \lambda x(t),\nonumber\\
		w(t)& = \lambda x(t),
		\end{align}
		where $\sigma(t) >0$ is the inflow rate at time~$t$,
		$x(t) \in [0,1]$ is the occupancy of the bottleneck,
		and~$\lambda>0$ controls the output flow~$w(t)$. 
		The rate of change of the occupancy is proportional to the inflow rate~$\sigma(t)$ and the
		\emph{vacancy}~$1-x(t)$, that is, as the   occupancy increases the effective entry rate decreases.
 
We   assume that the inflow is periodic with period~$T\geq 0$, i.e. $\sigma(t+T)=\sigma(t)$ for all~$t\ge0$. The occupancy~$x(t)$ (and thus also~$w(t)$) 
entrains, as the  system  is  
contractive~\cite{sontag_cotraction_tutorial,LOHMILLER1998683}. 
In other words, for any initial condition~$x(0)\in[0,1]$
the solution~$x(t)$ converges to a unique~$T$-periodic solution denoted~$x_\sigma$ and thus~$w $
converges to a~$T$-periodic solution~$w_\sigma$. 

The outflow of the bottleneck is the 
 input into  a Hurwitz    positive  linear system:
\begin{align}\label{linear_system}
\dot z &= Az + b w ,\nonumber \\
y & =c^T z,
\end{align}
where $A \in \mathbb R^{n\times n}$  is Hurwitz and
 Metzler and~$ b,c \in \mathbb R^n_+$ (see Fig.~\ref{f.cascade}).
  It is clear that for a~$T$-periodic~$\sigma(t)$, all trajectories of the cascade 
	converge to a unique trajectory~$(x_\sigma(t),z_\sigma(t))$ with~$x_\sigma(t)=x_\sigma(t+T)$
	and~$z_\sigma(t)=z_\sigma(t+T)$.

\begin{figure}
	\includegraphics[width=\columnwidth]{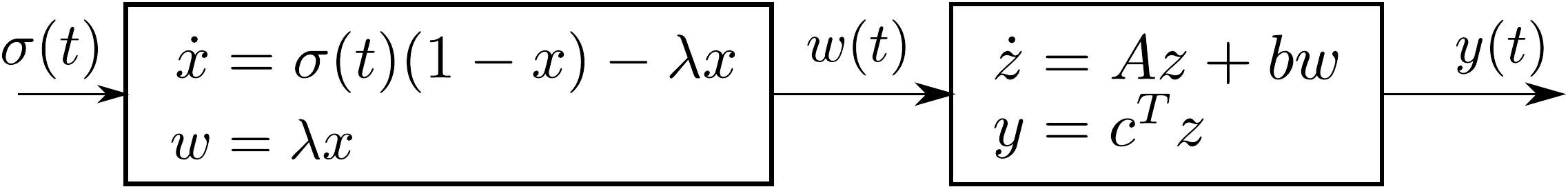}
	\caption{Cascade  system:  the bottleneck is feeding a positive linear system.} \label{f.cascade}
\end{figure}
Our goal is to compare the average (over a period) of~$y_\sigma(t)$
for various~$T$-periodic inflows. 
To make  a meaningful comparison, 
we consider inflows that have a fixed mean~$\bar \sigma>0$, i.e \begin{equation}\label{average}
\frac 1T \int_0^T \sigma(t) \diff t
 = \bar \sigma.
\end{equation}
The objective is maximize the gain of the system from~$\sigma$ to~$y$,
 i.e to maximize~$\int_0^T y_\sigma(t) \diff t $ for inputs with mean~$\bar \sigma$.

The trivial periodic inflow rate is the constant rate~$\sigma(t)\equiv \bar \sigma$. 
Here, we compare the outflow  for this constant inflow 
with that obtained for  
an inflow that switches  between two values~$\sigma_1$ and~$\sigma_2$ 
such that $\sigma_2>\bar \sigma >\sigma_1>0$.
  In other words,   $\sigma(t)\in \{\sigma_1,\sigma_2\}$   is periodic and satisfies~\eqref{average}.  

As  an application, consider  the  
 traffic system depicted in Fig.~\ref{f.traffic}.  There are two flows of vehicles with different rates~$\sigma_1,\sigma_2$ (e.g., cars and trucks)
 each moving in a separate road  and   joining 
 into a two-lane highway. This can be done in two ways. The
 first  is to place traffic lights at the end of each road, and switch between them before entering the highway as in Fig.~\ref{f.traffic}(a). The periodic traffic light signal~$\sigma(t)$ switches between the two flows, 
hence~$\sigma(t) \in \{\sigma_1,\sigma_2\}$. The second strategy is to have each road  constricted to a single lane, and then each joining the corresponding lane in the highway as in Fig.~\ref{f.traffic}(b). Hence, the inflow rate is constant
and equal to~$ (\sigma_1+\sigma_2)/2$. In both cases, the occupancy~$x(t)$ of the highway is modeled  by~\eqref{eq:con}. For a proper comparison, we require~$\tfrac 1T \int_0^T \sigma(t) \diff t =   (\sigma_1+\sigma_2)/2$ as discussed before.

\begin{figure*}[t!]
    \centering
    \begin{subfigure}[t]{0.475\textwidth}
        \centering
        \includegraphics[width=\columnwidth]{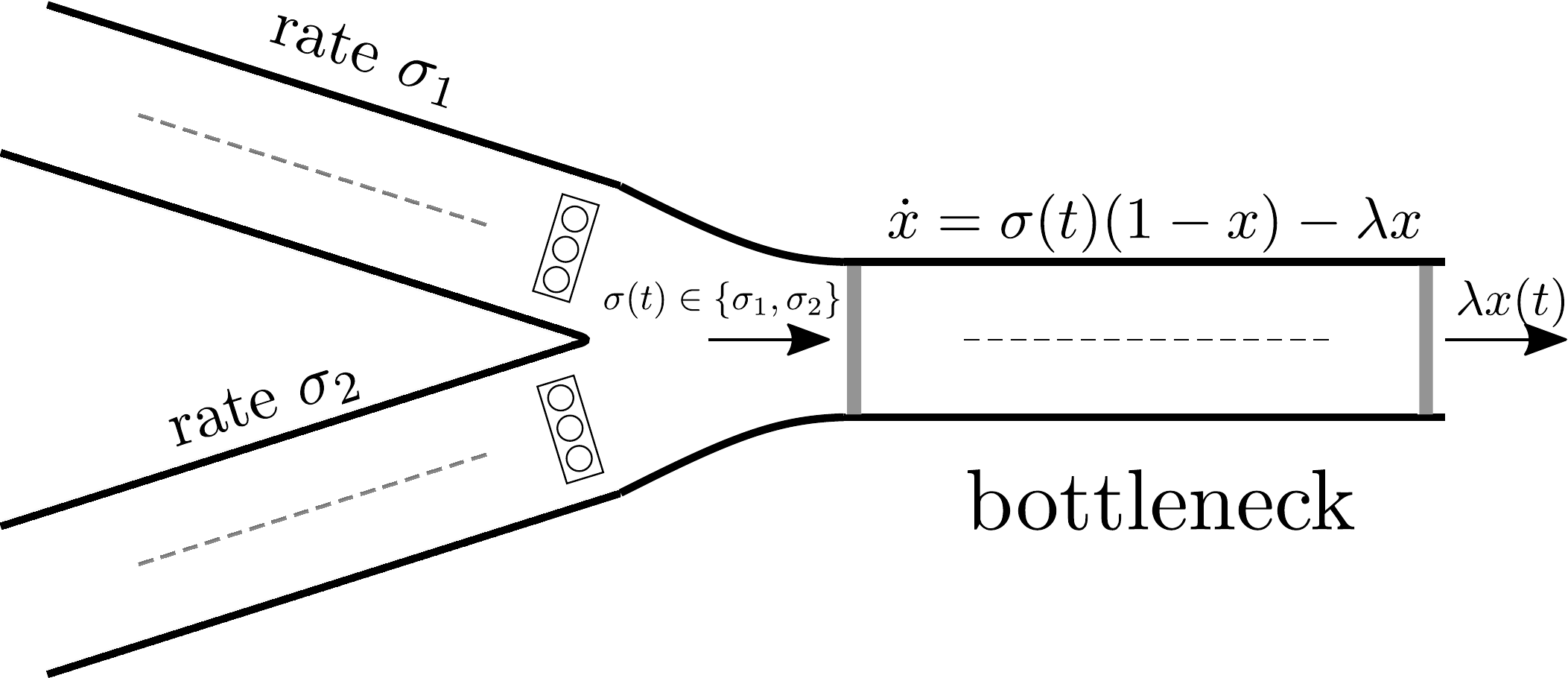}
        \caption{}
    \end{subfigure}%
    ~ 
    \begin{subfigure}[t]{0.475\textwidth}
        \centering
        \includegraphics[width=\columnwidth]{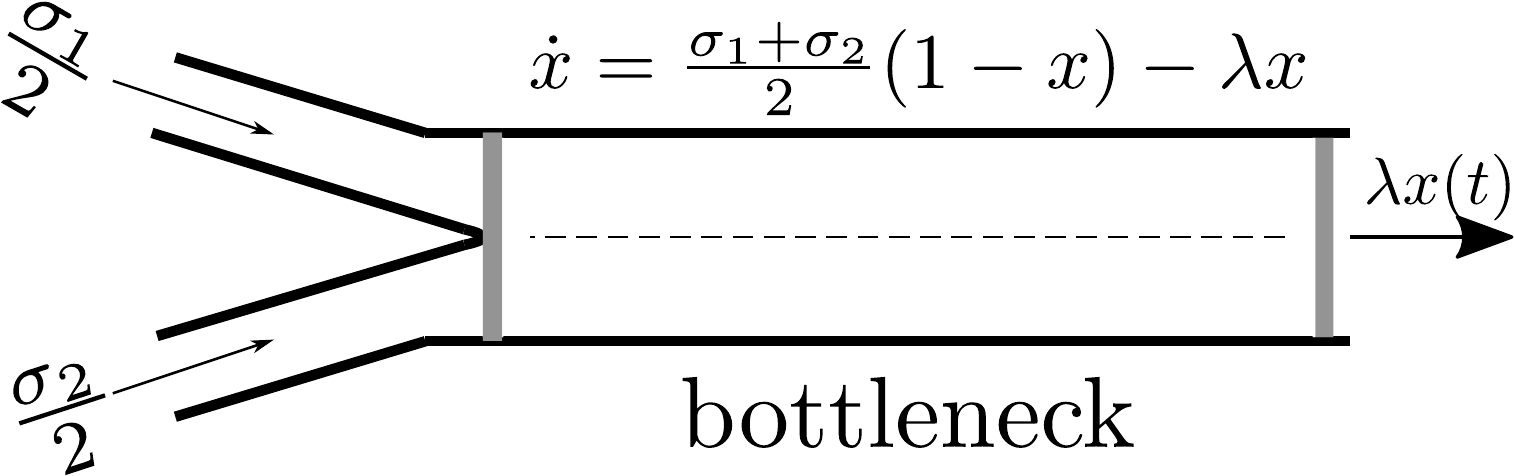}
        \caption{}
    \end{subfigure}
    \caption{Traffic system application illustrating the two strategies.
		Here $x(t) \in[0,1]$  denotes the occupancy of the bottleneck at time~$t$. 
	(a)~The inflow rate switches via periodically-varying
	traffic lights between two flows with rates $\sigma_1,\sigma_2$. At each time, either vehicles in the upper lane or vehicles in the lower lane can enter the bottleneck, but not both. 
	(b)~The double lane of each flow is restricted to a single lane and connected directly to the corresponding lane in the bottleneck. } \label{f.traffic}
\end{figure*}

%

	Another application  is the ribosome flow model~(RFM)~\cite{reuveni2011genome}
	which is 
	  a deterministic model 
	for ribosome  flow along the mRNA molecule.
	It can be derived via a  dynamic mean-field
	approximation of a fundamental model from statistical physics called
	the totally asymmetric simple exclusion process~(TASEP) \cite{TASEP_tutorial_2011,
		solvers_guide}.  In~\acrshort{tas} particles hop randomly
	along a chain of ordered sites. A site can be either   free or contain a single particle.
	Totally asymmetric means that the flow is unidirectional, and simple exclusion means that a particle can only hop into a free site. This models the fact that two particles cannot be in the same place at the same time. Note that this generates an indirect coupling between the particles. In particular, if a particle is delayed at a site for a long time then the particles behind it cannot move forward and thus 
	a ``traffic jam'' of occupied sites may evolve. There is considerable interest in the evolution
	and impact of   
	traffic jams of various ``biological machines'' in the cell (see, e.g.~\cite{Ross5911,ribo_q_2018}).
	
	 The~RFM is a nonlinear compartmental model with~$n$ sites.
	The state-variable~$x_i(t)$, $i=1,\dots,n$,
	describes the normalized density of particles
	at site~$i$ at time~$t$, so   
	 that~$x_i(t)=0$ [$x_i(t)=1$] means that site~$i$ is completely
	empty [full] at time~$t$. The state-space is thus the closed unit cube~$[0,1]^n$. The model includes~$n+1$  parameters~$\lambda_i$, $i=0,\dots,n$, 
	where~$\lambda_i>0$ controls the transition rate from site~$i$  to site~$i+1$. 
	In particular,~$\lambda_0$ [$\lambda_n$] is called the initiation  [exit] rate.
	
		The RFM dynamics is described by~$n$ first-order ODEs:
	\begin{align}\label{eq:rfm}
	\dot x_k &= \lambda_{k-1} x_{k-1}  (1-x_k ) 
		 -  \lambda_k x_k  (1-x_{k+1} ), \; k=1,\dots,n,
	\end{align} 
	where we define~$x_0(t)\equiv 1$ and~$x_{n+1}(t)\equiv 0 $.
	
	In the context of   translation, the~$\lambda_i$'s
	depend on various biomechanical properties for example
	the abundance of tRNA molecules that deliver the amino-acids to the ribosomes.
	A recent paper suggests  that cells   vary their~tRNA abundance in order to
	control the translation rate~\cite{Torrenteaat6409}. 
	Thus, it is natural to consider the case where the rates are in fact time-varying. 
	Ref.~\cite{entrainment} proved, using the fact that the~RFM
	is an (almost) contractive system~\cite{cast_book}, that if
  all the rates are jointly~$T$-periodic then every solution of the~RFM converges to a unique~$T$-periodic solution~$x^*(t) \in (0,1)^n$.

	Consider the RFM with a time-varying initiation
	rate~$\lambda_0(t)=\sigma(t)$
	and constant~$\lambda_1,\dots,\lambda_n$ such that~$\lambda_i \gg \sigma(t)$ for all~$i\geq 1$ and all~$t$. Then we can expect that the initiation rate becomes the bottleneck rate and thus~$x_i(t)$, $i=2,\dots,n$, converge to values that are close to zero, suggesting that~\eqref{eq:rfm} can be simplified to 
		\begin{align}\label{eq:rfm_l}
	\begin{split}
	\dot x_1 &=  (1-x_1 )\sigma  - \lambda_1 x_1  , \\
	\dot x_i&= \lambda_{i-1} x_{i-1}  -  \lambda_i x_{i},\;\; i\in \{2,\dots,n\}, 
	\end{split}
	\end{align} 
	which has the same form as the cascade in Fig.~\ref{f.cascade}.

	The remainder of this 
	note  is organized 
	as follows. Section~\ref{sec:formulation}
	presents a  formulation of the problem.
	Section~\ref{sec:main} describes our main results for the bottleneck system only, with
	the proof given in  Section~\ref{sec:proof}.
 Section~\ref{sec:casc} shows that the same result can be generalized to the cascade.   

%
%
%
	\section{Problem Formulation for the Bottleneck~System}\label{sec:formulation}
 Fix~$T\geq 0$. For any~$T$-periodic  function~$f:\R_+ \to \R$,  
let~$\ave (f) : =\frac{1}{T}\int_0^T f(s)\diff s$. 
Consider first only the bottleneck
model \eqref{eq:con}. 	 
A~$T$-periodic inflow rate~$\sigma(t)$ 
   induces	a unique 
attractive~$T$-periodic trajectory that we denote by~$x_\sigma(t)$ 
	 and thus a unique~$T$-periodic
	production rate~$w_\sigma (t):=\lambda x_\sigma (t)$. Thus, after a transient the
	  average production rate is~$\ave( {w_\sigma})$. 
	Now consider the bottleneck system
	with a constant inflow rate~$\bar\sigma$. 
Each trajectory converges to a unique steady-state~$\bar\sigma/(\lambda+\bar \sigma)$,
	and thus to a production rate~$\bar w:=\lambda \bar\sigma/(\lambda+\bar \sigma)$. 
	
	The question we are interested
	in is: \emph{what is the relation between~$\bar  {w}   $
	and~$\ave( w_\sigma)$}?
	To make this a  
	``fair'' comparison we consider  inflow rates from the set~$S_{\bar\sigma,T}$  that includes all 
	 the  admissible rates~$\sigma(t)$  that 
	are~$T$-periodic and satisfy~$\ave(\sigma) = \bar\sigma$. \rev{We note that this problem is related to \emph{periodic optimal control} \cite{bittanti73,gilbert77} where the goal is to find an optimal control~$u$ under the constraint that~$x(T,u)=x(0)$, yet
	without the additional constraint  enforcing that the average on the control is fixed.}

	We call~$ \sup_{\sigma  \in S_{\bar\sigma,T}}   {\ave(w_\sigma)}/{\bar w}, $ where the~$\sup$ is with respect to all the (non-trivial) rates in~$S_{\bar\sigma,T}$, 
	the \emph{periodic gain} 
	of the bottleneck  over~$S_{\bar\sigma,T}$.
	
	Indeed, one can naturally argue that the average production rate over a period,
	rather than the instantaneous  value, 
	is the  relevant quantity.  Then a periodic gain~$>1$ 
	implies that we can ``do better'' using periodic rates. A periodic gain  one implies that we do not  lose  nor  gain
		anything with respect to the constant rate~$\sigma(t)\equiv \bar\sigma$.
	A periodic gain~$<1$ implies that 
	for any (non-trivial) periodic rate the average production rate is lower than the one obtained for constant
	rates. This implies that
	entrainment  always incurs a cost, as the production rate for constant rates  is higher. 
	
	\rev{In this paper, we study the periodic gain with respect to inflows that belong to a subset~$B_{\bar\sigma,T} \subset S_{\bar\sigma,T}$ that is defined as follows.} Fix~$\sigma_2>\bar\sigma > \sigma_1>0$ such that~$ (\sigma_1+\sigma_2)/2=\bar\sigma$. Let $\sigma(t)$ be any measurable function 
	satisfying~$\ave(\sigma)=\bar \sigma$ and~$\sigma(t) \in \{\sigma_1,\sigma_2\}$, or $\sigma(t) \equiv \bar\sigma$. We will actually
	allow $\sigma$ to be a combination of the two   under a specific condition to be explained below. 
	
	Since the system entrains to a periodic signal, it is sufficient to consider it on the compact interval $[0,T]$. The entrained occupancy is periodic, hence we need to enforce the condition $x(0)=x(T)$. 
	
	For $\sigma(t) \in \{\sigma_1, \sigma_2\}$ the bottleneck system
	switches between two linear systems:
	\begin{align*}
	\dot x  =  \sigma_1 - (\sigma_1+\lambda) x ,\quad
	\dot x  =  \sigma_2 - (\sigma_2+\lambda) x.
	\end{align*}

	If $\sigma(t) \equiv \bar \sigma $ then 
	\begin{align}\label{ode_constant}
	\dot x = \tfrac{\sigma_1+\sigma_2}2 (1-x) - \lambda x,
	\end{align}
	 adding the constraint $x(0)=x(T)$, this implies that $x(t) \equiv \tfrac {\bar\sigma}{\bar\sigma + \lambda}$. 
	 
	 In general, we can consider $\sigma(t) \in \{ \sigma_1, \bar \sigma, \sigma_2\}$ but to have a meaningful combination the occupancy must be constant when $\sigma(t)=\bar\sigma$. In other words, $\sigma(t) = \bar\sigma$ implies that $x(t) =  \tfrac {\bar\sigma}{\bar\sigma + \lambda}$.  Hence the set of admissible inflows is defined as follows:
	\begin{align*} 
	B_{\bar\sigma,T}  &:= \{ \sigma(t)\! \in \! \{\sigma_1,\bar \sigma,\sigma_2\} |\textstyle \ave(\sigma)= \bar\sigma, \sigma(t)\!=\!\sigma(t+T), \\ &   \text{ and } 
		\sigma(t)=\bar\sigma \text{ implies that }   x(t)=\bar\sigma/(\lambda+\bar\sigma) \}. \end{align*}
	 
	 Our objective is to study the following quantity:
	 \[ J(\sigma(t)) := \ave(\lambda x_\sigma) /( \frac{\lambda \bar\sigma}{\bar\sigma+\lambda}),\]
	 over $\sigma(t) \in B_{\bar\sigma,T}  $. 
	\rev{ Here $\frac{\lambda \bar\sigma}{\bar\sigma+\lambda}$ is the outflow for the constant input $\sigma(t)=\bar \sigma$, and $\ave(\lambda x_\sigma)$ is the outflow   along the unique globally attracting periodic orbit corresponding to the periodic switching strategy.}

\section{ Periodic Gain of the  Bottleneck  System}\label{sec:main}
	The main result of this section can be stated as follows:
	\begin{theorem}\label{th1}
	Consider the bottleneck system \eqref{eq:con}. If $\sigma \in 
	B_{\bar \sigma,T}$ then $J(\sigma ) \le 1$.
			\end{theorem}
	Thus,   a constant inflow cannot be outperformed by a periodic one.

\begin{example} \label{ex:1}
Consider a bottleneck system with~$\lambda=1$
and switched inflows~$\sigma(t)$
 satisfying~$\ave(\sigma)=1$.
Note that for the 
 constant inflow~$\sigma(t)\equiv 1$ the 
 output converges to~$w_\sigma (t)\equiv 1/2 $, so
$\ave(w_\sigma)=1/2$.
Fig.~\ref{f.sim}(a)  depicts a histogram of the averaged 
outflow   for  randomly generated switching signals in~$B_{\bar \sigma,T}$.
 The mean performance for the  switching laws 
 is~$20\%$ lower than that achieved by the constant inflow.  

It follows from averaging theory~\cite{khalil}[Ch.~10] or
using the  Lie-Trotter product formula~\cite{Hall2015}
that it is possible to approximate the effect of a constant inflow \rev{with any desired accuracy}
using a \rev{sufficiently} fast switching law, but such a switching law
is not practical in most applications. 

Fig.~\ref{f.sim}(b) shows a scatter plot of the average outflow
for a bottleneck system
with~$\lambda=1$ and 
different values of~$\bar\sigma$. It may be seen that
it is harder to approach the performance of the constant inflow for higher~$\bar\sigma$.
\end{example}

\begin{figure}
	\centering
	\includegraphics[width=\columnwidth]{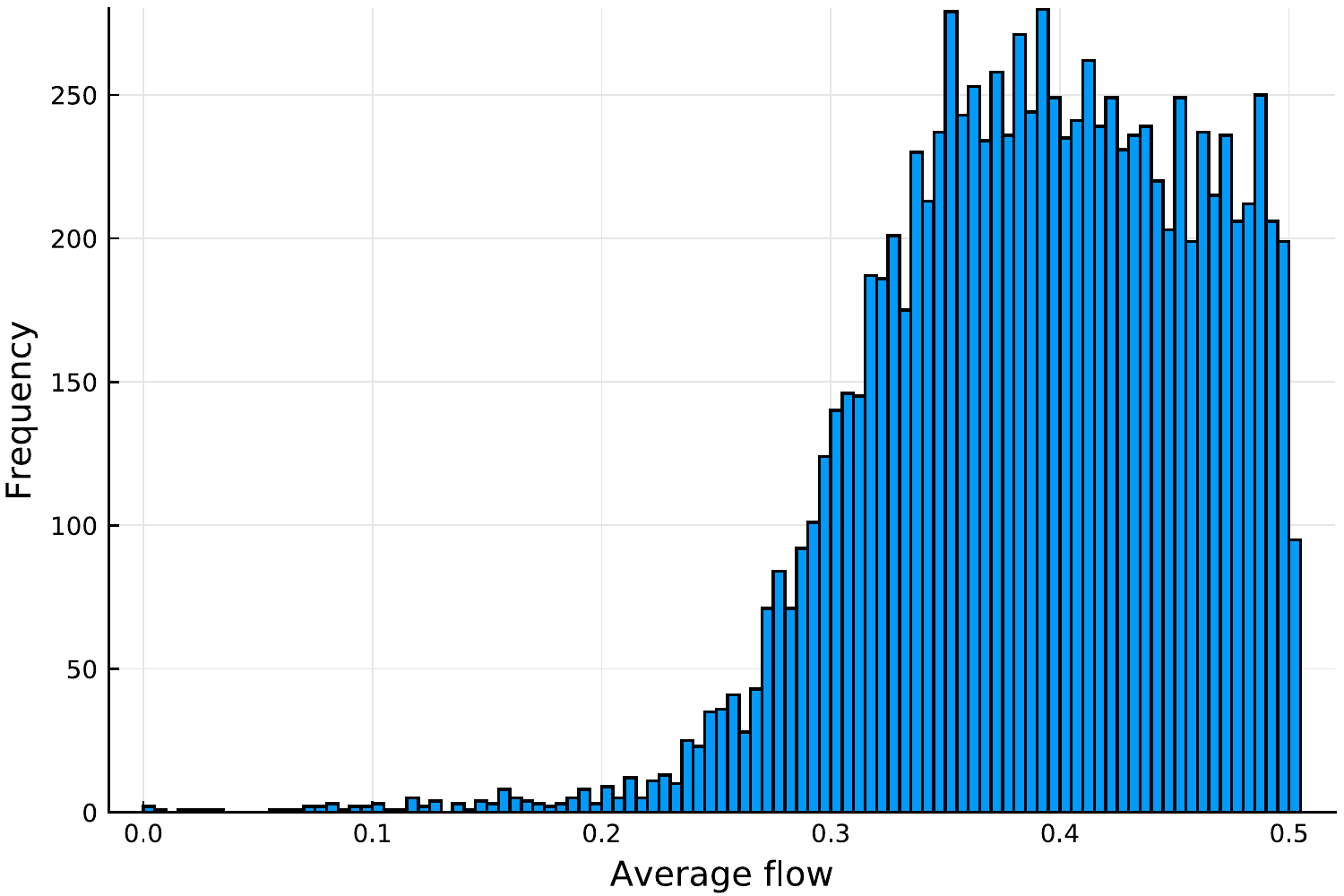}\\
	(a) \\
		\includegraphics[width=\columnwidth]{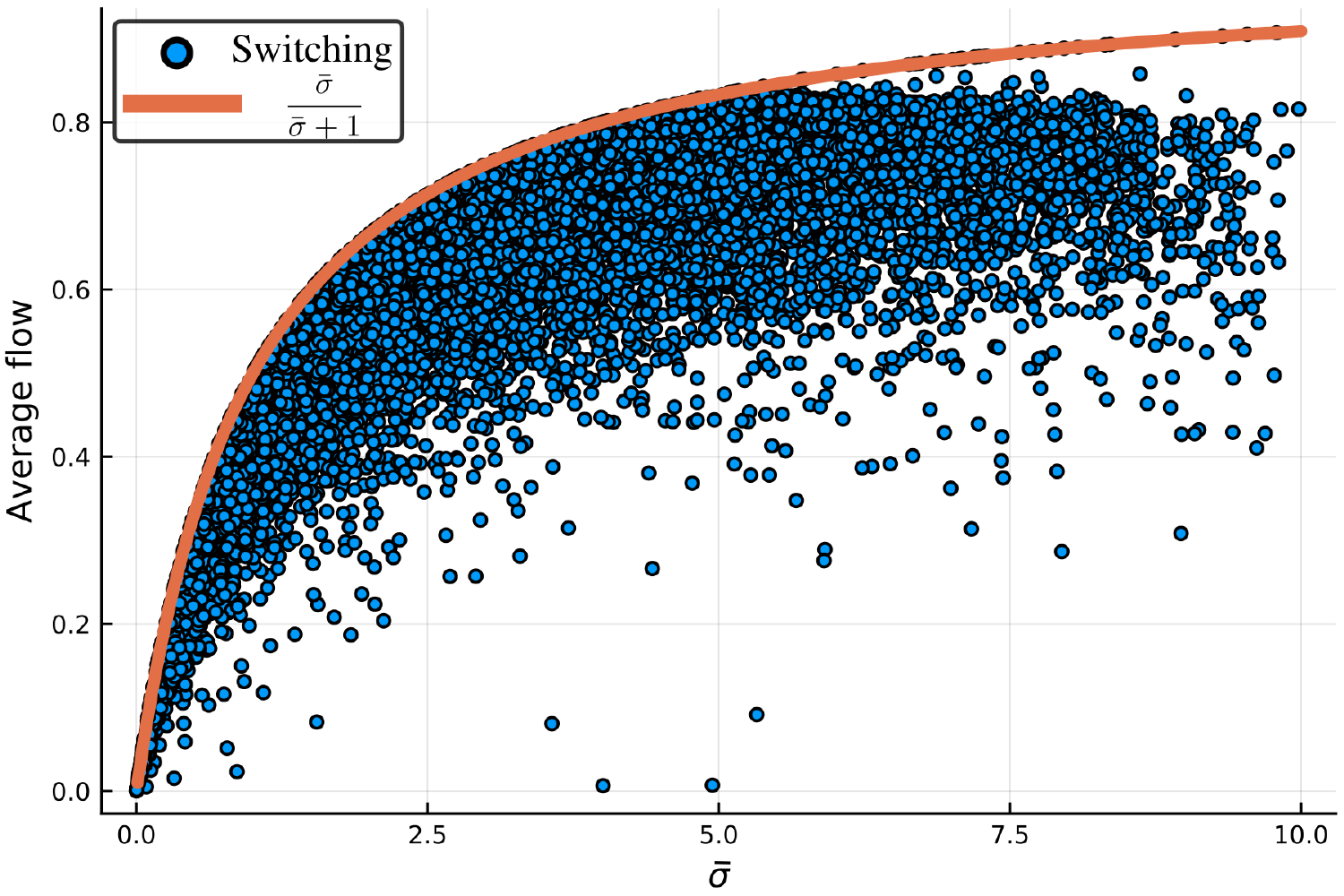}\\
		(b) \\
		\caption{Performance of the switching policy: (a) Histogram of achievable average outflow rate with $\bar\sigma=1$, $\lambda=1$. (b) Scatter plot  with~$\lambda=1$.  }
			\label{f.sim}
\end{figure}

	\section{Proof of Theorem \ref{th1}}\label{sec:proof}
	
	It is useful to parametrize the class of admissible inflows for a given average $\bar\sigma$ as follows:
	\begin{equation}\label{control_paramterization}
	\sigma(t) = \bar \sigma + \varepsilon \alpha(t),
	\end{equation}
	where $\alpha(t)$ is measurable,~$\ave(\alpha)=0$, 
	 $\alpha(t) \in \{-1,0,1\}$ for almost all~$t\in[0,T]$,
	and~$0\le \varepsilon\le \bar\sigma$.  
	\rev{Note that
	$\alpha(t)=1$  [$\alpha(t)=-1$]   corresponds
	to $\sigma(t)=\sigma_1$  [$\sigma(t)=\sigma_2$], while $\alpha(t)=0$ corresponds to the constant inflow $\sigma(t)=\bar \sigma=(\sigma_1+\sigma_2)/2$.} 
	Recall that for  every choice of~$\sigma(t)$ we let $x_\sigma$ denote the unique
	solution of~\eqref{eq:con} that satisfies~$x(0)=x(T)$.  
	
	\subsection{Finite Number of Switchings}
	  A set $E\subset [0,T]$ is said to be \emph{elementary} if it can be written as a \emph{finite} union of open, closed, or half-open intervals. For a given~$\alpha$,
		define~$E^+:=\{t|\alpha(t)=1\}$, $E^-:=\{t|\alpha(t)=-1\}$, and~$E^0:=\{t|\alpha(t)=0\}$. Then $\alpha(t)$ is said to have a finite number of switchings if $E^+$, $E^-$, and~$E^0$ are elementary sets.
	
	We are ready to state the next proposition:
	
	\begin{proposition}\label{th.finite}
		Suppose that~$\sigma(t)$ satisfies~\eqref{control_paramterization} 
		with~$\varepsilon>0$, 
		$\alpha(t)$ has   a finite number of switchings,
		and~$\mu(E^0)<T$\rev{, where $\mu$ denotes the Lebesgue measure. }
		Then,  $
		J(\sigma ) <1.
	$
	\end{proposition}
	
	To prove this we require three  simple  auxiliary results that we state
	as lemmas for easy reference.
	\begin{lemma}
		 Consider the scalar system~$\dot x(t) = a - b x(t)$ with~$x(t_0)=x_0$. Pick~$t_1 \geq t_0$, and
		let~$x_1:=x(t_1)$. Then:
		\begin{equation} b	\int_{t_0}^{t_1} x(t) \diff t = a (t_1-t_0) +x_0-x_1.
		\label{eq:intx}
		\end{equation}
	\end{lemma}
	\begin{proof}
	The solution at time~$t$  of the  scalar equation   satisfies   
	\begin{align}\label{eq:oifp}
	b x(t)=   b e^{-b(t-t_0)}x_0 +  a 	\Big( 1 - e^{-b(t-t_0)}  \Big) ,
	\end{align}
	and integrating yields
	\[
	b \int_{t_0} ^{t} x(t) \diff t=a(t_1-t_0)+ b x_0 - b x_0 e^{-b(t-t_0)}  - a\Big( 1 - e^{-b(t-t_0)}  \Big) 
	\]
	and combining this with~\eqref{eq:oifp} for~$t=t_1$ gives~\eqref{eq:intx}. 
	\end{proof}

	\begin{lemma} \label{lem:1} For any~$a,b>0$  we have  
		\begin{align}
		\frac{(1-e^{-a})(1-e^{-b})}{1-e^{-(a+b)}} < \frac{a b}{a+b}.
		\end{align}
	\end{lemma}
	\begin{proof} Let $f(a,b):=\frac{1-e^{-(a+b)}}{(1-e^{-a})(1-e^{-b})} - \frac 1a - \frac 1b$. The inequality is proved if we show that $f(a,b)>0$ for all~$a,b>0$. Note that $\lim_{(a,b)\downarrow (0,0)} f(a,b) =0$.
		Differentiating~$f$ with respect to $a$ yields
		$\frac{\partial f}{\partial a}=\frac 1{a^2}+\frac 1{2-2\cosh(a)}.$
		Using the Taylor series of $\cosh(a)$ we   find that $2\cosh(a)-2 > a^2$,
		so~$\partial f/\partial a>0$. Similarly, $\partial f/\partial b>0$. Hence, $f$ increases in all directions in the positive orthant.
	\end{proof}
	%

	\begin{lemma}\label{lem:bound}
		 Consider the scalar systems~$\dot x_i(t) = a_i - b_i x_i(t)$, $i=1,2$, with~$b_i>0$.  
		Suppose that there exist~$v,w \in \R$ and~$t_1,t_2 > 0$ such that
		for~$x_1(0)=v$ and~$x_2(0)=w$, we have~$x_1(t_1)=w$ and~$x_2(t_2)=v$. Then  
		\begin{align}\label{eq:wmv}
		w-v =  (c_1-c_2) \frac{  ( 1-e^{-b_1 t_1} )(1-e^{-b_2 t_2} )   }{ 1-e^{-b_1 t_1-b_2 t_2} } , 
		\end{align}
		where~$c_i:=a_i/b_i$, $i=1,2$. If, furthermore,~$c_1>c_2$ then
		\begin{align}\label{eq:bloy}
		w-v < (c_1-c_2)  b_1 t_1 b_2 t_2  / (b_1 t_1+b_2 t_2 ) . 
		\end{align}
		\end{lemma}
	\begin{proof}
	We know that $w=e^{-b_1 t_1} v +c_1 (1-e^{-b_1 t_1} )$
	and that~$v=e^{-b_2 t_2} w +c_2 (1-e^{-b_2 t_2} )$. Combining these equations yields
	\begin{align*}
	w&=c_1 (1-e^{-b_1 t_1} )+e^{-b_1 t_1} \left( c_2 (1-e^{-b_2 t_2} ) + e^{-b_2 t_2} w \right ),\\
	v&=c_2 (1-e^{-b_2 t_2} ) +e^{-b_2 t_2} \left (  c_1 (1-e^{-b_1 t_1} ) + e^{-b_1 t_1} v  \right ) ,
	\end{align*}
	and this gives~\eqref{eq:wmv}. If~$c_1>c_2$ then Lemma~\ref{lem:1}  yields~\eqref{eq:bloy}.
	\end{proof}

	Going back to our problem, note that the system \eqref{eq:con} with an input $\sigma$ 
	given in~\eqref{control_paramterization} is a
	\emph{switched linear system}~\cite{liberzon_book}
	which switches between three linear systems in the
	form~\rev{$\dot x=a_{z}-b_z x$, $z \in \{-,0,+\}$ (see Fig.~\ref{fig:main})}, 
	with
		\begin{equation}
\label{eq:ci}
	\begin{alignedat}{3}
	a_{-}&:= \bar\sigma -\varepsilon& ,\quad &b_{-}:= \bar \sigma +\lambda-\varepsilon,\\
	a_{0}&:= \bar\sigma   ,&\quad  &b_{0}:= \bar \sigma +\lambda ,\\
	a_{+}&:= \bar\sigma +  \varepsilon & ,\quad  &b_{+}:= \bar \sigma +\lambda + \varepsilon,
	\end{alignedat}
	\end{equation}
	corresponding to~$\alpha(t)=-1$, $0$, and~$1$, respectively. 
	We refer to the corresponding arcs as a~$z$-arc, with~$z\in\{- ,0, +\}$.  
	Note that~$b_+>b_0>b_{-}>0$. 
	Letting~$c_i:=a_i/b_i$,  we also have
	\begin{align}\label{eq:c1c2}
	c_+-c_{- }
	&= \frac{2 \lambda \varepsilon} { ( \bar \sigma +\lambda+\varepsilon ) ( \bar \sigma +\lambda-\varepsilon ) } >0.  
	\end{align}

Recall that along a $0$-arc the solution satisfies~$x(t)\equiv  c_0  $.  \rev{
Since~$x(T)=x(0)$, the other arcs form a  loop with a finite number of sub-loops. Observe that the $(+)$-arcs and the $(-)$-arcs can be paired:  if a $(+)$-arc traverses from~$x(t)=x_i^-$ to~$x(t+t^+_i)=x_i^+$, then there must be a~$(-)$-arc     that  traverses back from~$x_i^+$ to~$x_i^-$ (see  Fig.~\ref{fig:main}). Hence any trajectory can be partitioned into $n$ arcs with $x_0^-,x_0^+,x_1^-,x_1^+,...,x_n^-,x_n^+>0$ as the switching points.  Let $t_i^+$ [$t_i^-$] be the time spent on  the $i$'th  $(+)$- arc [$(-)$-arc], respectively.  }
Combining this with the assumption that~$\alpha(t)$
has a finite number of switches and~\eqref{eq:intx} implies that
\begin{align*}
\int _0^T x(t)\diff t &=  \int_{t\in E^0} x(t)\diff t+ \int_{t\in E^+} x(t)\diff t+ \int_{t\in E^{-}} x(t)\diff t\\
	&= c_0 \mu(E^0)  + \sum_{i=1}^{n} \left ( 
	c_{+} t_i^+ +(x_i^- - x_i^+ )/b_{+} \right ) \\
	& + \sum_{i=1}^{n} \left( c_{- } t_i^- +( x_i^+- x_i^- )/b_{-}\right) .
	\end{align*}
		Thus,
\begin{align*}
\int _0^T x(t)\diff t 
	&= c_0 \mu(E^0)   +c_+ \mu(E^+)   +c_- \mu(E^-) \\
   &+   (\frac{1}{b_-} -\frac{1}{b_+}  )   \sum_{i=1}^{n} (x_i^+-x_i^- )  .
	\end{align*}
	It follows from~\eqref{eq:bloy} that 
$
		x_i^+-x_i^-  < (c_+ - c_-) \frac{ b_+ t_i^+ b_- t_i^-  }{b_+ t_i ^+ + b_- t_i^- } , 
$
	so 
	 \begin{align}\label{eq:poute}
\int _0^T x(t)\diff t 
	&= c_0 \mu(E^0)  +c_+ \mu(E^+)   +c_- \mu(E^-)\nonumber \\
   &+     (c_+ - c_-)  ( {b_+} - {b_-}  ) 
	\sum_{i=1}^n \frac{  t_i^+  t_i^-  }{ b_+ t_i ^+ + b_- t_i^- } .
	\end{align}
Let~$t_i:=t_i^+ +t_i^-$ and~$\delta_i:=t_i^+-t_i^-$. Then
\begin{align}\label{eq:macc}
\frac{  t_i^+  t_i^-  }{ b_+ t_i ^+ +  b_- t_i^- } &=\frac{  (t_i+\delta_i)(t_i-\delta_i)  } { 4(b_0 t_i+\varepsilon \delta_i  )}\nonumber  \\ 
&=\frac{t_i^2-\delta_i^2}{4b_0 (  t_i+(\varepsilon \delta_i/b_0 ) )} \nonumber \\
& < \frac{1}{4b_0}(t_i- (\varepsilon \delta_i/b_0 ) ), 
\end{align}
where the last inequality follows from the fact that~$b_0^2=(\bar \sigma+\lambda)^2>\varepsilon^2$. 
Thus,
\begin{align*}
\sum_{i=1}^n \frac{   t_i^+   t_i^-  }{b_+  t_i ^+ + b_- t_i^- }&< 
\frac{1}{4b_0} \sum_{i=1}^n t_i\\
&=   (\mu(E^-)+\mu (E^+))/(4 b_0),
\end{align*}
and plugging this in~\eqref{eq:poute} yields 
	 \begin{align}\label{eq:ecier}
\int _0^T x(t)\diff t 
	&< c_0 \mu(E^0)  +c_+ \mu(E^+)   +c_- \mu(E^-)  \\
   &+      (c_+ - c_-)  ( {b_+} - {b_-}  ) 
	 (\mu(E^-)+\mu (E^+))/(4b_0).\nonumber
	\end{align}
Since~$\ave(\alpha)=0$,  $\mu(E^+)-\mu(E^-)=0$, and combining this with the fact that~$\mu(E^+)+\mu(E^-)+\mu(E^0)=T$ yields
\[
\mu(E^+)=\mu(E^-)=(T-\mu(E^0)) /2.
\]
Plugging  this and~\eqref{eq:ci} in~\eqref{eq:ecier} and simplifying 
yields 
$
  \int_0^T x(t) \diff t< \frac{ \bar \sigma}{\bar \sigma +\lambda}T , 
$
and this completes the proof of Prop.~\ref{th.finite}.

We note in passing that one advantage of our explicit approach is that 
by using a more exact analysis in~\eqref{eq:macc} it is possible to derive 
exact results on the ``loss'' incurred by using a non-constant inflow.

	\setlength{\unitlength}{0.22mm}
	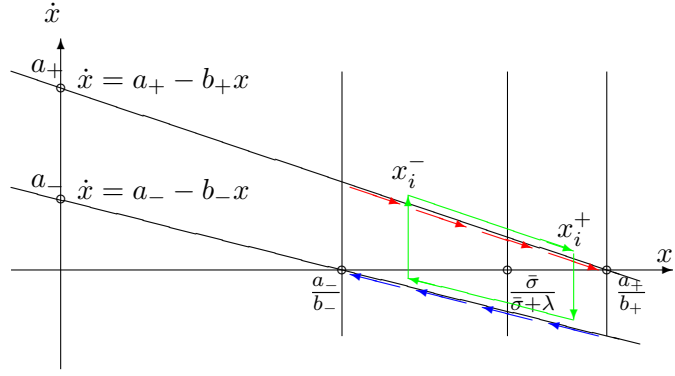
\begin{figure}[t]
		\centering
		\begin{picture}(400,230)
		\put(0, 60){\vector(1,0){400}}
		\put(30, 0){\vector(0,1){200}}
		
		\put(390,65){$x$}
		\put(20,210){$\dot x$}
		\put(230,115){$x^-_i$}
		\put(330,80){$x^+_i$}
		
		\put(0, 110){\line(4,-1){380}}
		\put(30,103){\circle{5}}
		\put(0, 180){\line(3,-1){380}}
		\put(30,170){\circle{5}}
		
		\put(200, 20){\line(0,1){160}}
		\put(300, 20){\line(0,1){160}}
		\put(360, 20){\line(0,1){160}}
		\put(300,60){\circle{5}}
		\put(200,60){\circle{5}}
		\put(360,60){\circle{5}}
		
		\put(300,42){$\frac{\bar \sigma}{\bar \sigma+ \lambda}$}
		\put(365,42){$\frac{a_+}{b_+}$}
		\put(180,42){$\frac{a_-}{b_-}$}
		\put(12,180){$a_+$}
		\put(12,110){$a_-$}
		\put(40,170){$\dot x = a_+ - b_+ x$}
		\put(40,103){$\dot x = a_- - b_- x$}
		
		\color{blue}
		\put(355, 20){\vector(-4,1){30}}
		\put(315, 30){\vector(-4,1){30}}
		\put(275, 40){\vector(-4,1){30}}
		\put(235, 50){\vector(-4,1){30}}
		
		\color{red} 
		\put(205, 110){\vector(3,-1){30}}
		\put(245, 96){\vector(3,-1){30}}
		\put(285, 83){\vector(3,-1){30}}
		\put(325, 70){\vector(3,-1){30}}
		
		\color{green} 
		\put(240, 105){\vector(3,-1){100}}
		\put(340, 70){\vector(0,-1){40}}
		\put(340, 30){\vector(-4,1){100}}
		\put(240, 55){\vector(0,1){50}}
		
		\end{picture}
		\caption{A trajectory of the switched system with $x(0)=x(T)$ traverses a loop in the $(\dot x,x)$-plane. 
		The upper [lower] line corresponds to~$\alpha(t)=1$ [$\alpha(t)=-1$].
		There is no line for the case $\alpha(t)=0$ since  we assume that the corresponding trajectory stays at a fixed point in the~$(x,\dot x)$ plane.}
		\label{fig:main}
	\end{figure}
	 
\subsection{Arbitrary Switchings}
	For two sets~$A,B$ let~$A \Delta B:=(A\setminus B)\cup (B\setminus A)$ denote the symmetric difference of the sets.
To prove Thm.~\ref{th1}, we need to consider a 
	measurable signal~$\alpha(t)$ in~\eqref{control_paramterization}.
	We use the following characterization of measurable sets:
\begin{lemma}\cite{kolmogorov} Let $E \subset [0,T]$. Then $E$ is measurable if and only if for every $\varepsilon>0$ there exists an elementary set $B_\varepsilon\subset [0,T]$ such that $\mu(E \Delta B_\varepsilon)< \varepsilon$. \label{kolm}
\end{lemma}
We improve on the lemma above, by the following:
\begin{lemma}\label{th.elementary} Let $E \subset [0,T]$. Then $E$ is measurable if and only if for every $\varepsilon>0$ there exists an elementary set $B_\varepsilon\subset [0,T]$ with $\mu(B_\varepsilon)=\mu(E)$ such that $\mu(E \Delta B_\varepsilon)< \varepsilon$.
\end{lemma}
\begin{proof}
	Sufficiency is clear. To prove necessity, pick~$\varepsilon>0$. 
	By Lemma~\ref{kolm}, there exists an elementary
	set~$B_{\varepsilon/2} \subset [0,T]$ 
	such that~$\mu(E \Delta B_{\varepsilon/2})< \varepsilon/2$. 
	  We can modify the intervals contained in $B_{\varepsilon/2}$ 
		by up to $\varepsilon/2$ to get $B_\varepsilon $ with $\mu(B_\varepsilon)=\mu(E)$
		and~$\mu(E \Delta B_{\varepsilon })< \varepsilon $.
\end{proof}

We generalize Prop.~\ref{th.finite} as follows:
\begin{proposition}\label{th.arbitrary}
	If $\sigma(t)$ is given as in~\eqref{control_paramterization} with 
	  $\alpha(t)$   measurable then~$J(\sigma ) \le 1$.
\end{proposition}
\begin{proof} 
  Let $E^0,E^+,E^-$ be defined as before. \rev{ For any~$i\geq 1$, let $\varepsilon_i:=2^{-i}$. By  Lemma~\ref{th.elementary},
	there exist elementary
	sets~$E_i^0, E_i^+,E_i^-$     such that 
	$\mu(E_i^0)=\mu(E^0)$
	and $\mu(E^0 \Delta E_i^0)<\varepsilon_i=2^{-i}$}, and similarly 
	for~$E_i^+, E_i^-$. Note that this implies that~$\mu(E_i^0)+\mu(E_i^+)+\mu(E_i^-)=T$. For every~$i\geq 1$
	define 
	\[\alpha_i(t): =  \begin{cases}
	1,& \text{ if } t \in E_i^+, \\
	0, & \text{ if } t \in E_i^0, \\
	-1 ,& \text{ if } t \in E_i^-. \end{cases}
	\]
	Then $\alpha_i(t) $ are elementary simple functions, 
	and~$\lim_{i\to\infty} \alpha_i(t) = \alpha(t)$ for all~$t\in[0,T]$. 
	For each~$i$, we can apply Prop.~\ref{th.finite} to the periodic
	solution~$x_{\alpha_i}$.
	Now the proof 
	follows  from Lebesgue's bounded convergence theorem~\cite{kolmogorov}.
\end{proof}

This completes the proof of Thm.~\ref{th1}.

\section{ Periodic Gain of the  Cascade System}\label{sec:casc}
We have shown  above that a constant inflow
outperforms any switched inflow 
 for the bottleneck system which is the first block in Fig.~\ref{f.cascade}.
We now show that the result can be generalized if we have a positive
linear system after the bottleneck. We first show that the periodic gain of a Hurwitz 
linear system does not depend on the particular periodic signal, but only on the DC gain of the system.

\begin{proposition} \label{prop:9}
	Consider a    SISO Hurwitz  linear system~\eqref{linear_system}. 
	Let~$w$ be a bounded measurable~$T$-periodic input.
	Recall that the output converges to a steady-state~$y_w$
	that is~$T$-periodic. Then
	\[
	 \frac{1}{T} \int_0^T y_w(t) \diff t
	= H(0) \frac{1}{T} \int_0^T  w(t) \diff t  ,
	\]
	where $H(s):=c^T(sI-A)^{-1}b$ is the transfer function of the linear system.
\end{proposition}
\rev{Thus, the periodic gain of a   linear system is the
same for any input.}

\begin{proof}
 Since $w$ is measurable and bounded,~$w \in L_2 ([0,T])$. Hence, it can be written as a Fourier series~\[
	w(t)=\ave( w)+ \sum_i a_i \sin (\omega_i  t +\phi_i ). 
	\]
	 The output of the linear system
	converges to:
$
	y_w(t)=    H(0) \ave(w)  +
		\sum_i  |H(j\omega_i)| a_i  \sin (\omega_i  t +\phi_i  + \mbox{arg}(H(j\omega_i) )).
		$
		Each sinusoid in the expansion has period~$T$, so
	$ \int_0^T y_w(t) \diff t = T  H(0) \ave( w ).$
\end{proof}

If~$H(0)=-c^TA^{-1}b\geq 0$ (e.g. when the linear system is positive) then 
and we can extend Thm.~\ref{th1} to the cascade in Fig.~\ref{f.cascade}:
\begin{theorem} \label{th:final}
	Consider the cascade of \eqref{eq:con} and~\eqref{linear_system}   in Fig.~\ref{f.cascade}.  If $\sigma \in B_{\bar \sigma,T}$  then $\frac 1T \int_0^T y_\sigma (t) \diff t  \le \tfrac{\lambda \bar\sigma}{\lambda+\bar \sigma} H(0).$
\end{theorem}
Thus, a constant inflow cannot be outperformed by a switched inflow. 
	\section{Conclusion}
	We analyzed the performance of a switching control for a specific type of~SISO
	nonlinear system that is relevant  in fields such as
	traffic control  and molecular biology. 
	
	\begin{figure}\centering
		\begin{minipage}[b]{0.59\columnwidth}		\centering
			\includegraphics[width=\columnwidth,height=3.5cm]{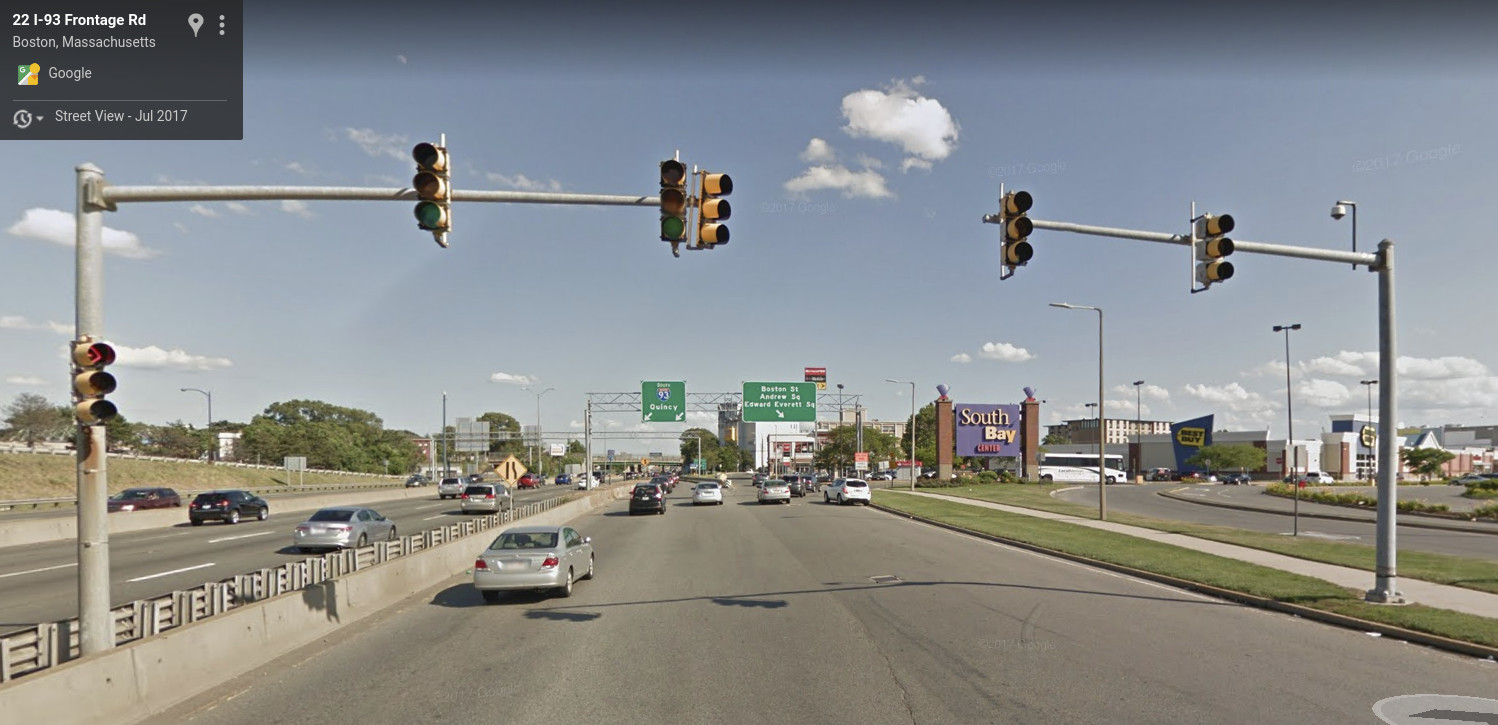} \\
			(a)
		\end{minipage}
	\begin{minipage}[b]{0.39\columnwidth}		\centering
		\includegraphics[width=\columnwidth, height=3.5cm ]{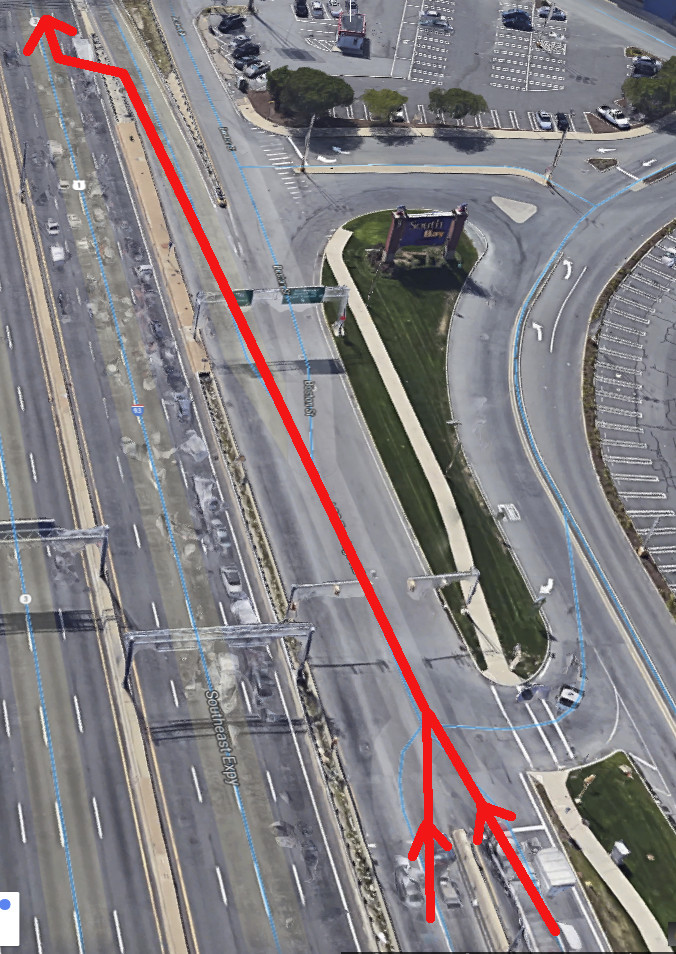} \\
		(b)
	\end{minipage}
\caption{Example of the switching design in Fig.~\ref{f.traffic}(a) 
located at an i-93 entrance in Boston, MA (42$^{\circ}$19'46.8"N 71$^{\circ}$03'39.7"W). (a)~The two traffic lights. (b)~The path taken by vehicles to join the i-93 road. (Copyright: Google Earth)}
\label{f.google}
\end{figure}
	We have shown
	that a constant inflow outperforms switching inflows with the same average. The performance of the switching inflow can be enhanced by faster switching, but this is not practical in most
	applications (e.g. traffic systems). Nevertheless, the switching inflow can be found in the real world (Fig.~\ref{f.google}), and our analysis 
	 shows that, for our model, such design involves
	a deterioration in the throughput of the system.
	
The periodic gain may certainly be larger than
	one for some nonlinear  systems and then  \rev{nontrivial}
	periodic inflows are better than the constant inflow. 
 For example, given the bottleneck system, let~$\xi(t):= 1- x(t)$. 
Then we obtain the dual system
	$
	\dot \xi(t) = \lambda(t) - (\lambda(t)+\sigma(t)) \xi(t). 
	$
	 Defining the output as~$w(t):=\lambda \xi(t)$ it follows 
	from Thm.~\ref{th1}   that here  a periodic switching cannot be outperformed
	by a constant inflow.
	
	A natural question then is how can one determine
	whether    the periodic gain of a given nonlinear system is larger or smaller than one.\rev{
	Other directions for further research include  analyzing
	the periodic gain of       important  nonlinear models like~TASEP
	and the~$n$-dimensional~RFM. A possible generalization would be to consider a bottleneck-input boundary linear hyperbolic PDE instead of
	a finite-dimensional system and its application to communication and traffic networks (see e.g.~\cite{espitia2017}).}

	\section*{Acknowledgments}	
	This work was partially supported by DARPA FA8650-18-1-7800, NSF 1817936, AFOSR FA9550-14-1-0060, Israel Science Foundation, and the US-Israel Binational Science Foundation.


	

\end{document}